\documentclass[12pt]{amsart}
\usepackage{amsthm}
\newtheorem{theorem}{Theorem}[section]
\newtheorem{lemma}[theorem]{Lemma}
\newtheorem{proposition}[theorem]{Proposition}
\theoremstyle{definition}
\newtheorem{definition}[theorem]{Definition}
\newtheorem{example}[theorem]{Example}
\theoremstyle{remark}
\newtheorem{remark}[theorem]{Remark}
\newtheorem{corollary}[theorem]{Corollary}

\counterwithin*{section}{part}

\usepackage{amssymb}
\usepackage{empheq}
\usepackage{stmaryrd}
\usepackage{enumerate}
\usepackage[shortlabels]{enumitem}
\usepackage{calc}
\usepackage{url}
\usepackage{comment}
\usepackage{mathabx}

\newcommand{\norm}[1]{\left\lVert#1\right\rVert}

\def\EE{\mathbb{E}}
\def\PP{\mathbb{P}}
\def\NN{\mathbb{N}}

\def\RR{\mathbb{R}}

\usepackage[left=2.0cm,%
                right=2.0cm,%
                top=2.5cm,%
                bottom=3.5cm,%
                headheight=12pt,%
                a4paper]{geometry}%

\begin{document}

\title[An abstract effective convergence theorem for stochastic processes]{An abstract effective convergence theorem for stochastic processes, with applications to stochastic approximation}

\author[M. Neri, N. Pischke, T. Powell]{Morenikeji Neri${}^{\MakeLowercase a}$, Nicholas Pischke${}^{\MakeLowercase b}$, Thomas Powell${}^{\MakeLowercase b}$}
\date{\today}
\maketitle
\vspace*{-5mm}
\begin{center}
{\scriptsize 
${}^a$ Department of Mathematics, Technische Universit\"at Darmstadt,\\
Schlossgartenstra\ss{}e 7, 64289 Darmstadt, Germany,\\ 
${}^b$ Department of Computer Science, University of Bath,\\
Claverton Down, Bath, BA2 7AY, United Kingdom,\\
E-mails: neri@mathematik.tu-darmstadt.de, \{nnp39,trjp20\}@bath.ac.uk}
\end{center}

\maketitle
\begin{abstract}
We provide a general theorem on the asymptotic behavior of stochastic processes that conform to a relaxed supermartingale condition. The distinguishing feature of our result is that it provides quantitative convergence guarantees at a much higher level of abstraction and generality than is typically seen in the stochastic approximation literature, formulated in particular in terms of a general modulus $\tau$ that, on an intuitive level, captures an effective variant of the \emph{uniqueness in expectation} of associated solutions. Our convergence rate is highly uniform, depending on very few data beyond $\tau$. We then demonstrate the utility of our result as a unifying framework by deriving new quantitative versions of several key concepts and theorems from stochastic approximation, including the Robbins-Siegmund theorem, Dvoretzky's convergence theorem, and the convergence of stochastic quasi-Fej\'er monotone sequences, the latter formulated in a novel and highly general metric context. Throughout, we isolate and discuss special cases of our results which allow for the construction of fast, and in particular linear, rates. Various applications of our results and our general methodology to stochastic approximation are discussed, and in particular explicitly derived in related work of the authors.
\end{abstract}

\section{Introduction}

Stochastic approximation methods are fundamental to modern science, with applications in optimization, statistics, machine learning, control theory, and many other areas. Accordingly, a vast array of methods have been studied across different settings and under various conditions, resulting in thousands of convergence proofs for stochastic algorithms. Underneath those many proofs lie a relatively small number of powerful techniques. and in can be extremely useful to isolate those techniques in the form of abstract, unified convergence theorems. Such results are significant in that they help us better understand the relationship between different stochastic algorithms and can therefore guide the development of new methods. Even more interestingly, \emph{strengthenings} of abstract theorems with, for example, quantitative information, propagate down to the many concrete methods whose convergence has been established using those theorems, providing new insights and results for both well-known and newly discovered stochastic algorithms.

Our main result is a quantitative convergence theorem of this kind (cf.\ Theorem \ref{thm:RatesGeneral}). It is intended as a modern, widely applicable supermartingale-type result, its characterising feature being the provision of explicit convergence rates at a high level of abstraction and generality. We show that our general theorem unifies several archetypal proof strategies from stochastic approximation, including classic supermartingale convergence (as embodied by the famous Robbins-Siegmund theorem \cite{robbins-siegmund:71:lemma}) under a stochastic regularity assumption (cf.\ Theorem \ref{thm:RS}), Dvoretzky's convergence theorem \cite{dvoretsky:56:stochastic} (cf.\ Theorem \ref{thm:quantDvo}), and the convergence of stochastic quasi-Fej\'er monotone sequences \cite{CP2015,Erm1969} for general metric spaces under a (stochastic) uniqueness assumption (cf.\ Theorem \ref{thm:Fejer}). While several of these results are qualitatively new, our main contribution rests on the quantitative convergence guarantees, both in mean and almost surely. Broadly speaking, these represent the most general effective rates possible for stochastic approximation problems with unique solutions (as discussed in more detail at the beginning of Section \ref{sec:fejer}), and go far beyond the reach of typical (and more ad-hoc) constructions of \emph{fast} (e.g.\ linear) rates as commonly given in the literature, which inevitably require much stronger assumptions on the problem and method. Nevertheless, we highlight how our framework can be readily refined to also guarantee such fast rates, still very generally and abstractly, in special cases that encompass those usually considered in the literature, and can thus be used to extend the reach of sharp estimates to new methods (cf.\ Theorems \ref{thm:RSfast} and \ref{thm:fejerFast}).

Our overall aim is to provide a general framework for obtaining quantitative convergence guarantees for stochastic approximation methods in settings where corresponding solutions are unique in a stochastic sense, one which is broadly applicable whenever convergence can be established using classic descent-based methods. Many immediate applications of our results are then possible, where by analysing existing convergence proofs, one can obtain either new or generalised convergence rates for known methods: A simple illustration of this can be found in Section \ref{sec:RM}, where we provide an abstract, quantitative convergence theorem for the Robbins-Monro algorithm, covering $\tau$-uniformly monotone operators in Hilbert spaces. 

However, the most effective use of our abstract theorems will be to establish novel convergence results for new methods under stochastic uniqueness assumptions. Indeed, this has already been exemplified in \cite{Pischke2025b,Pischke2026}, where the second author has applied the results given below to new algorithms. Concretely, \cite{Pischke2025b} extends a stochastic variant of the proximal point algorithm due to Bianchi \cite{Bia2016} to the general metric context of Hadamard spaces \cite{BH1999}, while \cite{Pischke2026} extends the recently introduced Busemann subgradient method of Goodwin, Lewis, L\'opez-Acedo and Nicolae \cite{GoodwinLewisLopezAcedoNicolae2026} to general stochastic minimization problems. In both cases, the associated results are already qualitatively novel and, moreover, quantitative convergence theorems for these methods under stochastic uniqueness assumptions are given (including fast rates), making heavy use of the methodology developed in this paper, in particular our formulation of stochastic quasi-Fej\'er monotonicity in general metric spaces. Further applications of the present methodology are given in the related work  \cite{PischkePowell2026}, as discussed further in Section \ref{sec:relatedWork} below (as well as Section \ref{sec:fejer}).

Many more applications of this kind are possible: for example to stochastic splitting methods for monotone inclusions (see e.g.\  \cite{CP2016,NVP2024,RVV2016}), all of which might potentially be lifted to broader geometric contexts; or applications that outright apply to more sophisticated geometric settings such as gradient descent-type algorithms on Riemannian manifolds (e.g.\ \cite{ACPY2012,Bon2013,TFBJ2018}). Beyond those, an indicative sample of the vast range of methods that utilise the Robbins-Siegmund theorem is provided in the survey article \cite{franci-grammatico:convergence:survey:22}, many of which are immediately open to generalisation using our results. Outside of concrete applications in stochastic optimization, one could also extend the present quantitative approach to one of the various unified convergence theorems already known in the literature, such as those found in the recent work \cite{LM2022}, or already in the older work \cite{JJS1994} (with the latter of particular note as it was originally conceived, and continues to be used, to provide a rigorous convergence proof for dynamic programming based learning algorithms). Further structures and concepts that could be captured through our abstract approach include coupled supermartingale convergence theorems (cf.\ \cite{WB2016,WFL2017}).

Finally, we suggest that our unifying perspective might also be relevant given the increasingly widespread (and sophisticated) use of computers in mathematics \cite{Tao2024}. Generalised convergence theorems have been particular targets for mechanisation in proof assistants (see e.g.\ the formalisation of Dvoretzky's theorem in \cite{VTSP2022}), and we conjecture that the abstract, quantitative results, if implemented within a proof assistant, would allow us to automatically generate convergence rates from suitably formalised convergence proofs. 

\subsection{Overview of the main results}

We now give a more technical overview of our main results. Consider a general nonnegative real-valued stochastic processes $(X_n)$ that satisfies a standard almost-supermartingale property:
\[
\EE[X_{n+1}\mid \mathcal{F}_n]\leq (1+A_n)X_{n}+C_n \text{ almost surely for all }n\in\mathbb{N},
\]
relative to a filtration $(\mathcal{F}_n)$ and nonnegative stochastic processes $(A_n),(C_n)$. As usual, the intuitive idea is that $X_n$ represents the distance between the element $x_n$ of a stochastic algorithm and some target point $z$, with the perturbation $1+A_n$ and error term $C_n$ arising in some way from stochastic noise. Together with natural conditions on the errors and the perturbations, we provide a convergence result for such processes, both in expectation and almost surely, in the presence of an additional approximation assumption
\[
\liminf_{n\to\infty}\EE[f(X_n)]=0.
\]
Here $f$ belongs to a class of well-behaved functions (defined for the first time in this paper) that \emph{slow down} the process $X_n$, and hence substantially weaken the stronger assumption $\liminf_{n\to\infty}\EE[X_n]=0$ that is typically required to establish convergence of $X_n$ towards zero (and hence of $x_n$ towards $z$). Our ability to slow down the process in this way is crucial for the applications that follow and one of the features that gives our main result its breadth.

The proof of this general result is rather elementary, relying only on fundamental properties of martingales and conditional expectations (notably Ville's inequality for nonnegative supermatingales). However, the result comes equipped with quantitative information that not only guarantees the convergence to zero of $f(X_n)$ in mean and of $X_n$ almost surely, but also provides rates in both cases in the form of explicitly and effectively constructed functions $\rho:(0,\infty)\to \mathbb{N}$ with
\[
\forall \varepsilon>0\ \forall n\geq \rho(\varepsilon)\left( \EE[f(X_n)]<\varepsilon\right)
\]
and $\rho':(0,\infty)^2\to\mathbb{N}$ with 
\[
\forall \lambda,\varepsilon>0\left( \PP(\exists n\geq \rho'(\lambda,\varepsilon)(X_n\geq \varepsilon))<\lambda\right),
\]
respectively. While rather abstract in this general formulation, in many concrete cases these functions will be immediately equivalent to ordinary (non-)asymptotic guarantees on $(X_n)$, as supplied e.g.\ also explicitly in our results on fast rates.

The motivation behind our approximation property for $f(X_n)$ is the observation from practice that in many concrete scenarios, this property (for suitable $f$) comes into existence through a decomposition into a similar property $\liminf_{n\to\infty}\EE[V_n]=0$ for a secondary process $(V_n)$ (representing an approximate solution property), which in many cases can be naturally and easily obtained, together with a type of regularity assumption relating $\EE[V_n]$ to $\EE[f(X_n)]$, which in this paper will be concretely represented through a modulus $\tau:(0,\infty)\to (0,\infty)$ satisfying
\[
\forall n\in\mathbb{N}\ \forall\varepsilon>0\left( \EE[V_n]<\tau(\varepsilon)\to \EE[f(X_n)]<\varepsilon\right).
\]
As we will argue, this modulus $\tau$ represents a powerful abstraction of many stochastic regularity or uniqueness assumptions from the literature.

Going from the abstract to the concrete, we then illustrate that the three distinct (but conceptually overlapping) approaches to convergence of stochastic approximation methods as discussed above, that is supermartingale convergence in the form of the Robbins-Siegmund theorem, Dvoretzky's convergence theorem and convergence of stochastic quasi-Fej\'er monotone sequences, are all (in respective quantitative variants) consequences of our abstract theorem, where the regularity moduli as detailed above naturally manifest themselves through commonly made assumptions in all these cases. This provides not only a uniform proof strategy for these three central approaches to stochastic approximation, but furthermore allows us to transfer our quantitative results to these specialised settings.

\subsection{Related work}\label{sec:relatedWork}

The results in this paper are partly inspired by the modern appreciation of Fej\'er monotonicity as one of the key unifying tools in fixed point theory and optimization, as especially due to Bauschke and Combettes \cite{BC2017,Combettes2001,Com2009} for deterministic methods, and in particular also Combettes and Pesquet \cite{CP2015,CP2019} in a stochastic setting.

Our paper is further inspired by the quantitative study of Fej\'er monotone sequences in the presence of general notions of regularity given in a deterministic setting by Kohlenbach, L\'opez-Acedo and Nicolae \cite{KLAN2019} (see also \cite{Pis2023,Pischke2025a}). In that context, our results can in particular be seen as a (partial) stochastic variant of the work \cite{KLAN2019}, focusing on uniqueness as a special case of the regularity notion studied therein. A study of stochastic variants of the broader regularity notion introduced in \cite{KLAN2019} is given in the work \cite{PischkePowell2026} by the second and third author (as will be discussed on a more technical level in Section \ref{sec:fejer} later on). In this regard, the present paper can be viewed as a companion article to \cite{PischkePowell2026}: Here, we work on a higher level of abstraction and consider multiple ways in which supermartingale-type convergence results can be formulated, whereas \cite{PischkePowell2026} restricts its attention to (slightly more restrictive variants of) stochastic quasi-Fej\'er monotonicity, for that under a much more liberal notion of regularity than the one investigated here. This however requires additional machinery, including careful implementations of results from measurable selection theory, which can be avoided under this paper's focus on uniqueness.

Finally, we highlight that the present paper is orthogonal to related work of the the first and third author \cite{NeriPowell2024,NeriPowell2026} on quantitative results for the martingale convergence and Robbins-Siegmund theorem as well as of the authors \cite{NeriPischkePowell2026} on similar such results for stochastic quasi-Fej\'er monotone processes in general metric settings (without any type of regularity assumption).\footnote{The results from \cite{NeriPowell2024,NeriPowell2026,NeriPischkePowell2026}, likewise the related works \cite{KLAN2019,Pis2023,Pischke2025a} on quantitative aspects of Fej\'er monotonicity in a deterministic setting, as well as the present work and its derived works \cite{Pischke2025b,Pischke2026,PischkePowell2026}, have been obtained using the logic-based methodology of proof mining \cite{Koh2008,Koh2019}. The probabilistic works are part of a series of recent applications of these methods to probability theory and stochastic optimization \cite{NeriOlivaPischke2026,NeriPischke2024}. Here, for example, this logical approach in particular influenced the way in which we formulated the various (stochastic) moduli. All the results and proofs given in this work are formulated in a way which avoids any reference to mathematical logic.} All of these only derive weaker oscillation-type bounds (or more generally so-called rates of metastability), which are the best type of quantitative results possible without making some type of regularity assumption as we do here. As such, the present paper  (as well as the work \cite{PischkePowell2026}) complements these works by showing that under such a regularity modulus, corresponding ``full'' rates of convergence can be explicitly and effectively defined in a stochastic context, and moreover, as mentioned before and as also discussed later in Section \ref{sec:fejer} (see also \cite{PischkePowell2026}), such general regularity moduli are essentially optimal.

\subsection*{Notation} For the remainder of the paper, we fix a probability space $(\Omega,\mathcal{F},\PP)$. We understand all measure-theoretic notions, such as random variables, (conditional) expectations $\EE$ and almost-sureness (which we abbreviate by a.s.), to be defined relative to it. All random variables, unless specified otherwise, are assume to be real-valued and (in-)equalities between random variables, if not specified otherwise, are understood to hold almost surely.

\section{A general theorem on rates of convergence for stochastic processes}\label{sec:genRates}

We begin with our general convergence theorem for stochastic processes that satisfy an almost-supermartingale condition. Crucially, this theorem is not just concerned with plain convergence, but also provides a rate of convergence for the associated process, both in expectation and almost surely. The theorem will then be instantiated in the subsequent sections to yield various applications in stochastic approximation, all of which inherit the convergence rates formulated here. 

We first discuss the main ingredients we require, the first of which is a well-known concentration inequality for nonnegative supermartingales due to Ville \cite{Ville1939} (see also \cite{Metivier1982} for a detailed overview of concentration inequalities for martingales):

\begin{lemma}\label{ville}
Let $(U_n)$ be a nonnegative supermartingale. Then for any $a>0$ we have
\[
\PP\left(\sup_{n\in\NN} U_n\geq a\right)\leq \frac{\EE[U_0]}{a}.
\]
\end{lemma}

Further, we will rely on Jensen's inequality for (conditional) expectations (the following formulation of which is taken from \cite{Kle2020}, see Theorems 7.9 and 8.20 therein):

\begin{lemma}\label{jensen}
Let $\mathcal{G}$ be a sub-$\sigma$-algebra of $\mathcal{F}$ and let $X$ be a nonnegative integrable random variable. Let $\varphi:[0,\infty)\to[0,\infty)$ be a measurable function. If $\varphi$ is convex, then
\[
\varphi(\EE[X\mid\mathcal{G}])\leq \EE[\varphi(X)\mid\mathcal{G}]\text{ a.s.}\text{ and }\varphi(\EE[X])\leq\EE[\varphi(X)],
\]
where if $\varphi$ is concave, then
\[
\varphi(\EE[X\mid\mathcal{G}])\geq \EE[\varphi(X)\mid\mathcal{G}]\text{ a.s.}\text{ and }\varphi(\EE[X])\geq\EE[\varphi(X)].
\]
\end{lemma}

Our convergence result and the construction of associated rates of convergence rests on a key descent condition of $\liminf$-type for the associated process, forcing it to have expectation below $\varepsilon$ for each $\varepsilon>0$ infinitely often. As already discussed in the introduction, this descent condition can be ``slowed down'' in the sense that only the expectation of the process under a suitable given function $f$ is expected to decrease asymptotically in this sense. This generality will be necessary for some of our applications that follow. Definition \ref{def:sicc} below specifies the relevant class of admissible functions $f$. 

\begin{definition}
A function $f:[0,\infty)\to [0,\infty)$ is called $\psi$-supermultiplicative for a function $\psi:[0,1]\to [0,1]$ if $f(xa)\geq f(x)\psi(a)$ for all $x\in [0,\infty)$ and $a\in [0,1]$ and $\psi$ satisfies $\psi(x)>0$ for $x>0$.
\end{definition}

\begin{definition}
A function $f:[0,\infty)\to [0,\infty)$ is called continuous at $0=f(0)$ with a modulus $\kappa:(0,\infty)\to (0,\infty)$ if $x<\kappa(\varepsilon)$ implies $f(x)<\varepsilon$ for all $x\in [0,\infty)$ and $\varepsilon>0$.
\end{definition}

\begin{definition}\label{def:sicc}
A function $f:[0,\infty)\to [0,\infty)$ is called s.i.c.c.\ (with moduli $\psi$ and $\kappa$) if it is
\begin{enumerate}
\item $\psi$-supermultiplicative for a function $\psi:[0,1]\to [0,1]$,
\item strictly increasing,
\item concave,
\item continuous, and $\kappa:(0,\infty)\to (0,\infty)$ is a modulus of continuity at $0=f(0)$.
\end{enumerate}
\end{definition}

Before we move on to our main theorem, we give examples of typical s.i.c.c.\ functions and discuss some closure properties enjoyed by that class.

\begin{example}\label{ex:cisc}
\begin{enumerate}
\item The function $x^q$ for $q\in (0,1]$ is s.i.c.c.\ with moduli $a^q$ and $\sqrt[q]{\varepsilon}$. In particular, $\sqrt{x}$ and $x$ are s.i.c.c. To see this, note that $x^q$ is clearly increasing and concave. Further, since $(xa)^q=x^qa^q$ for any $x,a\geq 0$, we get that $x^q$ is $a^q$-supermultiplicative. Lastly, $x^q$ is clearly continuous and since $x< \sqrt[q]{\varepsilon}$ implies $x^q<\varepsilon$, it immediately follows that $\sqrt[q]{\varepsilon}$ is a modulus of continuity for $x^q$ at $0=0^q$.
\item The function $\log_c(1+x)$ for $c>1$ is s.i.c.c.\ with moduli $a$ and $c^\varepsilon-1$. To see this, note that $\log_c(1+x)$ is increasing and concave. Further, $\log_c(1+x)$ is clearly continuous and since $x< c^\varepsilon-1$ implies $\log_c(1+x)<\varepsilon$, the function $c^\varepsilon-1$ is a modulus of continuity for $\log_c(1+x)$ at $0=\log_c(1)$. Lastly, $\log_c(1+x)$ is also $a$-supermultiplicative since we have $(1+x)^a\leq 1+xa$ for all $x\geq 0$ and all $a\in [0,1]$, so that $a\log_c(1+x)\leq \log_c(1+xa)$ for all such $x,a$.
\end{enumerate}
\end{example}

\begin{proposition}\label{prop:closure}
Let $f$ be s.i.c.c.\ with moduli $\psi$ and $\kappa$ and let $g$ be s.i.c.c.\ with moduli $\psi'$ and $\kappa'$. Then:
\begin{enumerate}
\item $\alpha f+\beta g$, given $\alpha,\beta>0$, is s.i.c.c.\ with moduli $\min\{\psi,\psi'\}$ and $\min\{\kappa(\varepsilon/2\alpha),\kappa'(\varepsilon/2\beta)\}$.
\item $f\circ g$ is s.i.c.c.\ with moduli $\psi\circ\psi'$ and $\kappa'\circ\kappa$.
\item $\min\{f,g\}$ is s.i.c.c.\ with moduli $\min\{\psi,\psi'\}$ and $\max\{\kappa,\kappa'\}$.
\end{enumerate}
\end{proposition}

Proposition \ref{prop:closure} follows immediately from some simple calculations, so we omit the proof. We now state and prove the main theorem of this section:

\begin{theorem}\label{thm:RatesGeneral}
Let $(\mathcal{F}_n)$ be a filtration of $\mathcal{F}$ and let $(X_{n})$, $(A_n)$ and $(C_n)$ be sequences of nonnegative, integrable real-valued random variables adapted to $(\mathcal{F}_n)$. Suppose that for all $n\in\NN$:
\[
\EE[X_{n+1}\mid \mathcal{F}_n]\leq (1+A_n)X_{n}+C_n\text{ a.s.}
\]
Also, suppose that there exist $K\geq 1$ and $\chi:(0,\infty)\to \NN$ satisfying $\prod_{i=0}^\infty (1+A_i)<K$ a.s.\ and $\sum_{i=\chi(\varepsilon)}^\infty \EE[C_i]<\varepsilon$ for all $\varepsilon>0$. Further, let $f:[0,\infty)\to [0,\infty)$ be s.i.c.c.\ with moduli $\psi$ and $\kappa$. Finally, suppose that $\varphi$ is a $\liminf$-modulus for $(f(X_{n}))$ in expectation in the sense that
\[
\forall\varepsilon>0\ \forall N\in\NN\ \exists n\in [N;\varphi(\varepsilon,N)]\left(\EE[f(X_{n})]<\varepsilon\right).
\]
Then $\EE[f(X_n)]\to 0$ with rate
\[
\rho(\varepsilon):=\varphi\left(\frac{\varepsilon \psi(K^{-1})}{2},\chi\left(\kappa\left(\frac{\varepsilon \psi(K^{-1})}{2}\right)\right)\right)
\]
and $X_n\to 0$ a.s.\ with rate $\rho'(\lambda,\varepsilon):=\rho(\lambda f(\varepsilon))$.
\end{theorem}
\begin{proof}
For any $n\in\mathbb{N}$, define
\[
U_{n}:=\frac{X_{n}}{B_{n-1}}+\EE\left[\sum_{i=n}^\infty \frac{C_i}{B_i}\mid \mathcal{F}_n\right]\text{ where }B_j:=\prod_{i=0}^j (1+A_i),
\]
with the convention that $B_{-1}:=1$. The stochastic process $(U_{n})$ is a nonnegative supermartingale, since for any $n\in\mathbb{N}$ we have
\begin{align*}
\EE[U_{n+1}\mid\mathcal{F}_n]&=\EE\left[\frac{X_{n+1}}{B_{n}}\mid\mathcal{F}_n\right]+\EE\left[\EE\left[\sum_{i={n+1}}^\infty \frac{C_i}{B_i}\mid \mathcal{F}_{n+1}\right]\mid\mathcal{F}_n\right]\\
&=\frac{\EE[X_{n+1}\mid\mathcal{F}_n]}{B_n}+\EE\left[\sum_{i={n+1}}^\infty \frac{C_i}{B_i}\mid \mathcal{F}_{n}\right]\\
&\leq \frac{X_{n}}{B_{n-1}}+\frac{C_n}{B_n}+\EE\left[\sum_{i={n+1}}^\infty \frac{C_i}{B_i}\mid \mathcal{F}_{n}\right]\\
&=U_{n},
\end{align*}
using that $B_j\geq 1$ for every $j$. Thereby $(f(U_{n}))$ is also a nonnegative supermartingale since for an arbitrary $n\in\mathbb{N}$, we have $\EE[f(U_{n+1})\mid\mathcal{F}_n]\leq f(\EE[U_{n+1}\mid\mathcal{F}_n])\leq f(U_{n})$ using Jensen's inequality (and that $f$ is continuous as well as concave) as well as that $f$ is monotone. Further, for any $n\in\mathbb{N}$, we get
\[
f(U_{n})\leq f\left(\frac{X_{n}}{B_{n-1}}\right)+f\left(\EE\left[\sum_{i=n}^\infty \frac{C_i}{B_i}\mid \mathcal{F}_n\right]\right)\leq f(X_{n})+f\left(\EE\left[\sum_{i=n}^\infty C_i\mid \mathcal{F}_n\right]\right)
\]
using that $f$ is subadditive (since $f$ is concave) and that $f$ is monotone together with $B_{j}\geq 1$ for every $j$. In particular, using Jensen's inequality (and so the concavity and continuity of $f$) again, we have
\[
\EE[f(U_{n})]\leq \EE[f(X_{n})]+f\left(\EE\left[\sum_{i=n}^\infty C_i\right]\right).\tag{$+$}\label{plus}
\]
Now, let $\varepsilon>0$ and choose
\[
n(\varepsilon)\in \left[\chi\left(\kappa\left(\frac{\varepsilon \psi(K^{-1})}{2}\right)\right);\varphi\left(\frac{\varepsilon \psi(K^{-1})}{2},\chi\left(\kappa\left(\frac{\varepsilon \psi(K^{-1})}{2}\right)\right)\right)\right]
\]
such that $\EE[f(X_{n(\varepsilon)})]<\varepsilon\psi(K^{-1})/2$. Let $m\geq n(\varepsilon)$ be arbitrary. Then
\[
\EE[f(U_{m})]\leq \EE[f(U_{n(\varepsilon)})]\leq \EE[f(X_{n(\varepsilon)})]+f\left(\EE\left[\sum_{i=n(\varepsilon)}^\infty C_i\right]\right)\leq \varepsilon \psi(K^{-1})
\]
using that $(f(U_{n}))$ is a supermartingale as well as \eqref{plus} and that $f$ is monotone and continuous at $0$  together with the defining property of $\chi$, which gives $\EE\left[\sum_{i=n(\varepsilon)}^\infty C_i\right]<\kappa\left(\varepsilon \psi(K^{-1})/2\right)$. Finally, we get that
\[
\psi(K^{-1})f(X_m)\leq  f\left(\frac{X_{m}}{K}\right)<f\left(\frac{X_{m}}{B_{m-1}}\right)\leq f(U_{m})
\]
as $B_{m-1}<K$, using also the monotonicity and $\psi$-supermultiplicativity of $f$, and after taking expectations we get
\[
\psi(K^{-1})\EE[f(X_m)]\leq \EE[ f(U_{m})]<\varepsilon \psi(K^{-1})
\]
and so $\EE[f(X_m)]<\varepsilon$. As $m$ was arbitrary, this yields that $\rho$ is a rate of convergence for $\EE[f(X_n)]\to 0$. For the rate that $X_n\to 0$ a.s., note that
\begin{align*}
\PP(\exists m\geq n(\varepsilon)(X_m\geq a))&\leq \PP(\exists m\geq n(\varepsilon)(U_{m}\geq a/K))\\
&\leq \PP(\exists m\geq n(\varepsilon)(f(U_{m})\geq f(a/K)))\\
&\leq \PP(\exists m\geq n(\varepsilon)(f(U_{m})\geq f(a)\psi(K^{-1})))\\
&\leq \frac{\EE[f(U_{n(\varepsilon)})]}{f(a)\psi(K^{-1})}
\end{align*}
where the first inequality follows from the fact that $X_{m}/K<X_{m}/B_{m-1}\leq U_{m}$ for all $m\in\mathbb{N}$ and the last inequality follows from Ville's inequality. This immediately implies that $X_n\to 0$ a.s.\ with rate $\rho'$.
\end{proof}

Disregarding the quantitative information in the above result, we in particular get the following qualitative theorem, applicable to a broad class of stochastic processes:

\begin{corollary}\label{cor:qualCorMain}
Suppose that $(X_n)$, $(A_n)$ and $(C_n)$ satisfy the supermartingale-type property of Theorem \ref{thm:RatesGeneral} and assume that $\prod_{i=0}^\infty (1+A_i)<\infty$ a.s., that $\sum_{i=0}^\infty \EE[C_i]<\infty$, and that $\liminf_{n\to\infty}\EE[f(X_n)]=0$ for some s.i.c.c.\ function $f$. Then $\EE[f(X_n)]\to 0$ and $X_n\to 0$ a.s.
\end{corollary}

\begin{remark}
The conclusion of $\EE[f(X_n)]\to 0$ in the above Theorem \ref{thm:RatesGeneral} is in general the best one can hope for when allowing for general s.i.c.c.\ functions $f$ other than the identity. Concretely, already in the case of $f=\sqrt{\cdot}$, it is not necessarily the case that we also have $\EE[X_n]\to 0$. To see this, consider the following example: Let $(Y_n)$ be a nonnegative i.i.d.\ stochastic process with $\EE[Y_n]=1$ but $\EE[\sqrt{Y_n}]=\eta\in (0,1)$ for all $n\in\mathbb{N}$ (arising e.g.\ naturally through a Bernoulli process with values $0$ and $2$, each with probability $\frac{1}{2}$). Defining $X_n:=\prod_{k=0}^nY_k$ as well as choosing $\mathcal{F}_n:=\sigma(Y_0,\dots,Y_n)$, i.e.\ $\mathcal{F}_n$ is the $\sigma$-algebra generated by $Y_0,\dots,Y_n$, we immediately obtain that $(X_n)$ is a martingale w.r.t.\ the filtration $(\mathcal{F}_n)$. As $(Y_n)$ is independent, we have
\[
\EE[\sqrt{X_n}]=\prod_{k=0}^n\EE[\sqrt{Y_k}]=\prod_{k=0}^n\eta\to 0.
\]
In particular, the conditions of Theorem \ref{thm:RatesGeneral} are trivially satisfied for $A_n:=C_n:=0$. However, for similar reasons we have $\EE[X_n]=\prod_{k=0}^n\EE[Y_k]=1$ for any $n\in\mathbb{N}$, i.e.\ $(\EE[X_n])$ does not converge to $0$.
\end{remark}

In many situations where Theorem \ref{thm:RatesGeneral} applies, the crucial $\liminf$-property for $(f(X_n))$ in expectation and its corresponding modulus arise through a composition of two other properties, quantitatively witnessed by corresponding moduli in an analogous way. Concretely, a $\liminf$-modulus $\varphi$ for $(f(X_n))$ in expectation comes into existence through a $\liminf$-modulus $\varphi'$ for an auxiliary sequence $(V_n)$ in expectation, i.e.
\[
\forall\varepsilon>0\ \forall N\in\NN\ \exists n\in [N;\varphi'(\varepsilon,N)]\left(\EE[V_n]<\varepsilon\right),
\]
together with a type of regularity modulus $\tau:(0,\infty)\to (0,\infty)$ connecting $(V_n)$ and $f(X_n)$ in expectation, by which we here concretely mean that 
\[
\forall n\in\mathbb{N}\ \forall\varepsilon>0\left( \EE[V_n]<\tau(\varepsilon)\to \EE[f(X_n)]<\varepsilon\right).
\]
In the presence of these two moduli, we immediately get that
\[
\varphi(\varepsilon):=\varphi'(\tau(\varepsilon),N)
\]
is a $\liminf$-modulus for $(f(X_n))$ in expectation. This rather abstract conception of stochastic regularity in particular underlies our notion of stochastic uniqueness introduced in Section \ref{sec:fejer} later on, and beyond that serves as the model for the related stochastic regularity moduli considered in \cite{PischkePowell2026}, as already discussed in the introduction (as well as later in Section \ref{sec:fejer}). 

In that context, a particularly useful result is the following, which (formulated in an abstract way) guarantees that whenever such a modulus is convex and increasing, a regularity property in mean as given above can be derived from an associated (and generally easier to establish) \emph{pointwise} regularity property. This result is also important in \cite{PischkePowell2026}, and we refer to Remark 3.9 therein for a discussion of the fact that the convexity of such a modulus can be largely guaranteed, by taking convex envelopes.

\begin{proposition}\label{pro:AStoEconvex}
Let $I$ be a non-empty index set and let $(X_i)_{i\in I}$, $(V_i)_{i\in I}$ be two families of nonnegative real-valued random variables with
\[
V_i\geq\tau(X_i)\text{ a.s.}
\]
for all $i\in I$, where $\tau:[0,\infty)\to [0,\infty)$ is a convex and strictly increasing function with $\tau(0)=0$. Assume that each $X_i$ is integrable. Then $\tau$ satisfies
\[
\forall i\in I\ \forall\varepsilon>0\left( \EE[V_i]<\tau(\varepsilon)\to \EE[X_i]<\varepsilon\right).
\]
\end{proposition}
\begin{proof}
Take $\varepsilon>0$ and $i\in I$ with $\EE[V_i]<\tau(\varepsilon)$. Thus $\EE[\tau(X_i)]<\tau(\varepsilon)$ by the above inequality. Using Jensen's inequality, as $\tau$ is convex, we have $\tau(\EE[X_i])\leq \EE[\tau(X_i)]<\tau(\varepsilon)$. As $\tau$ is increasing, we have $\EE[X_i]<\varepsilon$.
\end{proof}

\begin{remark}\label{rem:regEquiv}
The above condition that $V_i\geq\tau(X_i)$ a.s.\ for all $i\in I$ can, for a strictly increasing $\tau$, be equivalently recognized as requiring that $\tau$ satisfies $V_i<\tau(\varepsilon)\rightarrow X_i<\varepsilon$ a.s.\ for all $i\in I$ and $\varepsilon>0$, and is hence equivalent to a pointwise variant of the regularity property in expectation. Similarly, one can show that such a $\tau$ satisfies a regularity property in mean in the form of $\EE[V_i]<\tau(\varepsilon)\to \EE[X_i]<\varepsilon$ for all $i\in I$ and $\varepsilon>0$, if, and only if, it induces a growth condition between $X_i$ and $V_i$ in mean, in the form of $\EE[V_i]\geq \tau(\EE[X_i])$ for all $i\in I$.
\end{remark}

\section{Applications to the Robbins-Siegmund theorem}\label{sec:RS}

The first application of our general result the seminal and widely used theorem of Robbins-Siegmund on supermartingale convergence:

\begin{theorem}[\cite{robbins-siegmund:71:lemma}]\label{thm:RS:original}
Let $(X_n)$, $(A_n)$, $(B_n)$ and $(C_n)$ be sequences of nonnegative integrable real-valued random variables adapted to a filtration $(\mathcal F_n)$. Assume that $\sum_{i=0}^\infty A_i<\infty$ and $\sum_{i=0}^\infty C_i<\infty$ a.s.\ and 
\[
\EE[X_{n+1}\mid\mathcal{F}_n]\leq (1+A_n)X_n-B_n+C_n
\]
a.s.\ for all $n\in\NN$. Then $(X_n)$ converges a.s.\ and $\sum_{i=0}^\infty B_i<\infty$ a.s.
\end{theorem}

As anticipated by the range of applications presented in the original work \cite{robbins-siegmund:71:lemma}, this result has become a fundamental tool for establishing the convergence of stochastic approximation algorithms, and we refer to the recent survey paper \cite{franci-grammatico:convergence:survey:22} for a broad overview of some of the many applications up to 2022.

In its full generality as presented above, the Robbins-Siegmund theorem does not admit a direct quantitative counterpart presenting effective ``full'' rates of convergence: In particular, it contains as a special case the standard Doob convergence theorem for $L^1$-supermartingales, where convergence can be arbitrarily slow.\footnote{\label{specker}A more precise way of formulating this point using the language of computability theory is that $L^1$-supermartingales inlcude, as trivial special case, monotone bounded sequences of real numbers, in which context well-known results from computable analysis due to Specker \cite{specker:49:sequence} rule out general computable rates of convergence.} The question of what quantitative information is possible in this \emph{general} case has been recently addressed in \cite{NeriPowell2026}, where a kind of oscillation bound is given. 

However, in many applications where the Robbins-Siegmund theorem is utilized to prove the convergence of some concrete stochastic method, one is actually in the situation that $B_n:= u_nV_n$ for $(u_n)$ a sequence of nonnegative control parameters with $\sum_{i=0}^\infty u_i=\infty$, from which we can additionally infer that $\liminf_{n\to\infty} \EE[V_n] = 0$. Moreover, the circumstances of the problem at hand allow often allow us to derive a modulus of regularity connecting the process $V_n$ with the process $X_n$ in the style of the preceding Section \ref{sec:genRates}. In these situations, we can apply Theorem \ref{thm:RatesGeneral} instead of the Doob martingale convergence theorem to establish convergence of $X_n$ towards zero, in mean and almost surely, together with explicit rates of convergence in both cases. This holds for several of the applications discussed in the original paper of Robbins and Siegmund \cite{robbins-siegmund:71:lemma} (including stochastic approximation and Dvoretzky's theorem, where the latter two will be discussed later), and is evidently the case in many of the hundreds of subsequent papers where supermartingale convergence has been applied.

To obtain our quantitative version of the Robbins-Siegmund theorem in the above circumstances, we require two preliminary quantitative results. The first is a quantitative version of slight strengthening of a lemma of Qihou \cite{Qih2001}, which itself is a non-stochastic analogue of the Robbins-Siegmund theorem. The first part of the following result appears e.g.\ as Lemma 5.31 in \cite{BC2017} and the bound on the series $\sum_{i=0}^\infty\beta_i$ appears as Theorem 3.2 in \cite{NeriPowell2026} (together with a quantitative result on the convergence of the sequence $(x_n)$ which, in the absence of the terms $(\beta_n)$, was already discussed in \cite{KL2004}).

\begin{lemma}\label{lem:qihou}
Let $(x_n)$, $(\alpha_n)$, $(\beta_n)$ and $(\gamma_n)$ be sequences of nonnegative reals with
\[
x_{n+1}\leq (1+\alpha_n)x_n-\beta_n+\gamma_n
\]
for all $n\in\mathbb{N}$. If $\prod_{i=0}^\infty(1+\alpha_i)<\infty$ and $\sum_{i=0}^\infty\gamma_i<\infty$, then $(x_n)$ converges and $\sum_{i=0}^\infty\beta_i<\infty$. Further, if $K,L,M>0$ satisfy $x_0<K$, $\prod_{i=0}^\infty(1+\alpha_i)<L$ and $\sum_{i=0}^\infty\gamma_i<M$, then $\sum_{i=0}^\infty\beta_i<L(K+M)$.
\end{lemma}

Second, we also require the following result on the asymptotic behavior of the summands of certain series.

\begin{lemma}\label{sumconv}
Let $(u_n)$, $(v_n)$ be sequences of nonnegative reals with $\sum_{n=0}^\infty u_nv_n< L$ where $L>0$ and $\sum_{n=k}^{\theta(k,b)} u_n\geq b$ for all $b>0$ and $k\in\NN$ where $\theta:\mathbb{N}\times(0,\infty)\to \NN$. Then $\liminf_{n\to\infty}v_n=0$ with 
\[
\forall\varepsilon>0\ \forall N\in\mathbb{N}\ \exists n\in [N;\theta(N,L/\varepsilon)](v_n<\varepsilon).
\]
\end{lemma}
\begin{proof}
Fix $\varepsilon>0$ and $N\in\NN$. If $v_n\geq \varepsilon$ for all $n\in [N;\theta(N,L/\varepsilon)]$, we would have
\[
L\leq \varepsilon\sum_{n=N}^{\theta(N,L/\varepsilon)} u_n\leq \sum_{n=N}^{\theta(N,L/\varepsilon)} u_nv_n\leq \sum_{n=0}^{\infty} u_nv_n<L
\]
which is a contradiction.
\end{proof}

The main result of this section is now a quantitative version of the Robbins-Siegmund theorem \cite{robbins-siegmund:71:lemma} adapted to use-cases in stochastic approximation, which follows as a relatively simple corollary of Theorem \ref{thm:RatesGeneral}.

\begin{theorem}\label{thm:RS}
Let $(X_n)$, $(V_n)$ and $(C_n)$ be sequences of nonnegative integrable real-valued random variables adapted to $(\mathcal{F}_n)$, and $(a_n)$, $(u_n)$ be sequences of nonnegative reals. Suppose that
\[
\EE[X_{n+1}\mid\mathcal{F}_n]\leq (1+a_n)X_n-u_nV_n+C_n\text{ a.s.}
\]
for all $n\in\NN$. Also, suppose that there exist $K,L,M>0$ satisfying $\EE[X_0]<L$, $\prod_{i=0}^\infty(1+a_i)<K$ and $\sum_{i=0}^\infty \EE[C_i]<M$ along with $\chi:(0,\infty)\to \NN$ and $\theta:\NN\times (0,\infty)\to \NN$ such that $\sum_{i=k}^{\theta(k,b)} u_i\geq b$ for all $b>0$ and $k\in\NN$ as well as $\sum_{i=\chi(\varepsilon)}^\infty \EE[C_i]<\varepsilon$ for all $\varepsilon>0$. Finally, suppose that $\tau:(0,\infty)\to (0,\infty)$ satisfies
\[
\forall \varepsilon>0\, \forall n\in\NN\left(\EE[V_n]<\tau(\varepsilon)\to \EE[f(X_n)]<\varepsilon\right)
\]
for some s.i.c.c.\ $f:[0,\infty)\to [0,\infty)$ with moduli $\psi$ and $\kappa$. Then $\EE[f(X_n)]\to 0$ with rate
\[
\rho(\varepsilon):=\theta\left(\chi\left(\kappa\left(\varepsilon'\right)\right),\frac{K(L+M)}{\tau(\varepsilon')}\right) \ \ \text{for} \ \ \varepsilon':=\frac{\varepsilon\psi(K^{-1})}{2}
\]
and $X_n\to 0$ a.s.\ with rate $\rho'(\lambda,\varepsilon):=\rho(\lambda f(\varepsilon))$.
\end{theorem}

\begin{proof}
Integrating both sides of the supermartingale property we obtain
\[
\EE[X_{n+1}]\leq (1+a_n)\EE[X_n]-u_n\EE[V_n]+\EE[C_n].
\]
Lemma \ref{lem:qihou} yields $\sum_{n=0}^\infty u_n\EE[V_n]\leq K(L+M)$. Therefore by Lemma \ref{sumconv} we have $\liminf_{n\to\infty}\EE[V_n]=0$ modulus $\varphi'(\varepsilon,N):=\theta(N,K(L+M)/\varepsilon)$ and thus a $\liminf$-modulus for $(f(X_n))$ is given by $\varphi(\varepsilon,N):=\varphi'(\tau(\varepsilon),N)$. Applying Theorem \ref{thm:RatesGeneral} with $A_n:=a_n$ gives the result.
\end{proof}

Theorem \ref{thm:RS} represents a general quantitative convergence result, applicable in typical situations in which the Robbins-Siegmund theorem appears, and whenever the underlying regularity property can be formulated in a quantitative way. Precisely because of its generality, the resulting convergence rates will not always be optimal under stronger conditions. However, in those cases, it will typically be possible to refine Theorem \ref{thm:RS} so that it gives us faster rates, under these stronger assumptions. We now illustrate this idea with a simple but meaningful instantiation, showing that if the regularity modulus is linear (as is, for example, the case for various examples as discussed in detail in \cite{PischkePowell2026}, as well as \cite{KLAN2019}) then we can also obtain linear rates of convergence under suitable assumptions on the other parameters. We first need a slight adaptation of a standard quantitative result about reals (see e.g.\ \cite{Nemirovski2009}):

\begin{lemma}\label{lem:ssLikeForQM}
Suppose that $(x_n)$ is a sequence of nonnegative reals such that for $c>1$, $d\geq 0$ and $r\in \NN\backslash\{0\}$, we have
\[
x_{n+1}\leq \left(1-\frac{c}{n+r}\right)x_n+\frac{d}{(n+r)^2}
\]
for all $n\in\NN$. Then for all $n\in\NN$:
\[
x_n\leq \frac{u}{n+r} \ \text{ for }\ u\geq \max\left\{\frac{d}{c-1},rx_0\right\}.
\]
\end{lemma}

\begin{proof}
We show the claim by induction. The case $n=0$ holds by definition and for the induction step we note that since $d\leq u(c-1)$, we have $d\leq u\left(c-\frac{n+r}{n+r+1}\right)$ for all $n,r$, and therefore
\begin{align*}
x_{n+1}&\leq \left(1-\frac{c}{n+r}\right)\frac{u}{n+r}+\frac{d}{(n+r)^2}\\
&\leq u\left(\left(1-\frac{c}{n+r}\right)\frac{1}{n+r}+\left(c-\frac{n+r}{n+r+1}\right)\frac{1}{(n+r)^2}\right)\\
&=u\left(\frac{1}{n+r}-\frac{1}{(n+r+1)(n+r)}\right)\\
&=\frac{u}{n+r+1}\left(\frac{n+r+1}{n+r}-\frac{1}{n+r}\right)\\
&=\frac{u}{n+r+1}
\end{align*}
which completes the induction.
\end{proof}

This gives the following result on fast rates for the Robbins-Siegmund theorem in the context of fast parameters and a linear regularity modulus $\tau(\varepsilon)=t\varepsilon$ for some $t>0$.

\begin{theorem}\label{thm:RSfast}
Let $(X_n)$, $(V_n)$ and $(C_n)$ be sequences of nonnegative integrable real-valued random variables adapted to $(\mathcal{F}_n)$, and $(a_n)$, $(u_n)$ be sequences of nonnegative reals. Suppose that
\[
\EE[X_{n+1}\mid\mathcal{F}_n]\leq (1+a_n)X_n-u_nV_n+C_n\text{ a.s.}
\]
for all $n\in\NN$. Also, let $K\geq 1$ and $L>0$ be such that $\prod_{i=0}^\infty (1+a_i)<K$ and $L\geq \EE[X_0]$ and assume that
\[
\EE[C_n]\leq d/(n+r)^2\text{ and }a_n+c/(n+r)\leq tu_n
\]
for all $n\in\mathbb{N}$ where $c>1$, $d\geq 0$ and $r\in\mathbb{N}\setminus\{0\}$. Finally, suppose we have $\EE[V_n]\geq t\EE[X_n]$ for all $n\in\mathbb{N}$ with $t>0$. Then
\[
\EE[X_n]\leq\frac{u}{n+r} \text{ and }\PP\left(\exists m\geq n(X_m\geq \varepsilon)\right)\leq \frac{1}{\varepsilon}\cdot\frac{K(u+2d)}{n+r}
\]
for all $n\in\mathbb{N}$ and $\varepsilon>0$, where $u\geq \max\left\{d/(c-1),rL\right\}$.
\end{theorem}
\begin{proof}
Integrating the inequality, we get
\begin{align*}
\EE[X_{n+1}]&\leq (1+a_n)\EE[X_n]-u_n\EE[V_n]+\EE[C_n]\\
&\leq (1+a_n-tu_n)\EE[X_n]+\EE[C_n]\\
&\leq \left(1-\frac{c}{n+r}\right)\EE[X_n]+\frac{d}{(n+r)^2}.
\end{align*}
Applying Lemma \ref{lem:ssLikeForQM} yields the rate for $\EE[X_n]$. For the second claim, we proceed similar to the proof of Theorem \ref{thm:RatesGeneral}. Concretely, note first that
\[
\EE[C_n]=\frac{d}{(n+r)^2}\leq \frac{2d}{(n+r)(n+r+1)}=2d\left(\frac{1}{n+r}-\frac{1}{n+r+1}\right)
\]
so that
\[
\sum_{i=n}^\infty\EE[C_i]\leq 2d\sum_{i=n}^\infty\left(\frac{1}{n+r}-\frac{1}{n+r+1}\right)=\frac{2d}{n+r}
\]
Now defining 
\[
U_{n}:=\frac{X_{n}}{b_{n-1}}+\EE\left[\sum_{i=n}^\infty \frac{C_i}{b_i}\mid \mathcal{F}_n\right], \text{ where }b_j:=\prod_{i=0}^j (1+a_i)
\]
with $b_{-1}:=1$, analogously to Theorem \ref{thm:RatesGeneral}, we have that $(U_n)$ is a supermartingale. Since 
\[
\EE[U_n]=\frac{\EE[X_n]}{b_{n-1}}+\sum_{i=n}^\infty\frac{\EE[C_i]}{b_i}\leq \EE[X_n]+\sum_{i=n}^\infty \EE[C_i]\leq \frac{u+2d}{n+r},
\]
using Ville's inequality, this yields 
\[
\PP\left(\exists m\geq n(X_m\geq \varepsilon)\right)\leq \PP\left(\exists m\geq n(U_m\geq \varepsilon/K)\right)\leq \frac{K}{\varepsilon}\cdot \EE[U_n]\leq \frac{1}{\varepsilon}\cdot\frac{K(u+2d)}{n+r}.\qedhere
\]
\end{proof}

\section{Applications to Dvoretzky's theorem}\label{sec:dvo}

A notable application of the Robbins-Siegmund theorem, as already highlighted in the original paper \cite{robbins-siegmund:71:lemma}, is a generalization of the seminal theorem of Dvoretzky \cite{dvoretsky:56:stochastic} in stochastic approximation to Hilbert spaces.\footnote{Dvoretzky later offered an alternative proof of this generalization in \cite{dvoretzky1986stochastic}, see also \cite{venter1966dvoretzky}.} Concretely, the following theorem is established in \cite{robbins-siegmund:71:lemma}:

\begin{theorem}[\cite{robbins-siegmund:71:lemma}, based on \cite{dvoretsky:56:stochastic}]\label{thm:dvo}
Let $X$ be a separable Hilbert space and let, for any $n\geq 1$, $T_n:X^n\to X$ be a Borel-measurable function such that there exists a point $z\in X$ along with sequences of nonnegative real numbers $(a_n),(b_n),(c_n)$ such that
\[
\norm{T_n(z_0,\dots,z_{n-1})-z}\le \max\{a_{n-1}, (1+b_{n-1})\norm{z_{n-1}-z} -c_{n-1}\}
\]
for all $z_0,\dots,z_{n-1}\in X^n$. Fix $X$-valued random variables $x_0$ and $(y_n)$ such that $\EE[y_{n}\mid \mathcal{F}_{n}]=0$ for each $n\in\mathbb{N}$, where $\mathcal{F}_n:= \sigma(x_0,y_0,\dots,y_{n-1})$, i.e.\ the $\sigma$-algebra generated by $x_0,y_0,\dots,y_{n-1}$. In $x_0$ and $(y_n)$, define the iteration
\[
x_{n+1}:= T_{n+1}(x_0,\dots,x_{n}) + y_n.
\]
Finally, suppose that $a_n \to 0$, $\sum_{n=0}^\infty b_n<\infty$, $\sum_{n=0}^\infty c_n = \infty$ and $\sum_{n=0}^\infty\EE[\norm{y_n}^2]< \infty$. Then $x_n\to z$ a.s.
\end{theorem}

A number of stochastic approximation algorithms take the form of the general iterative procedure of the above theorem, most notably the Kiefer-Wolfowitz scheme \cite{kiefer-wolfowitz:52:scheme} and (many variants of) the Robbins-Monro procedure \cite{Robbins1951}. Even further, Dvoretzky's result presented general conditions on the parameters of the previously introduced iterative scheme that unified many stochastic approximation results in the literature.

We now provide a quantitative version of Dvoretzky's theorem by combining an analysis of the proof of this result given in \cite{robbins-siegmund:71:lemma} with the quantitative theorems we established so far. We make a few careful choices in our treatment of the assumptions of Theorem \ref{thm:dvo}. First, as in the proof presented in \cite{robbins-siegmund:71:lemma}, the sequence $(c_n)$ can be assumed, without loss of generality, to also satisfy $\sum_{n=0}^\infty c_n^2 < \infty$, and accordingly, the result presented below will feature (a quantitative rendering of) this additional assumption. Second, note that as an immediate consequence of the assumptions Theorem \ref{thm:dvo} we have $\EE[\norm{x_n}^2]<\infty$ for all $n\in\mathbb{N}$. While we could calculate a bound on these means in terms of the other parameters recursively in $n\in\mathbb{N}$, for simplicity we will just directly assume bounds $L_n>0$ on these means as primitive inputs.

Our quantitative version of Dvoretzky's theorem now takes the form of Theorem \ref{thm:quantDvo} below, and the proof heavily relies on the fact that our general result, as well as our quantitative Robbins-Siegmund theorem, allows for processes slowed down by a s.i.c.c.\ function, which here will be instantiated by the square root. Only the presence of that function are we able to produce a corresponding modulus of regularity.

\begin{theorem}\label{thm:quantDvo}
Let $X$ be a separable Hilbert space and let, for any $n\geq 1$, $T_n:X^n\to X$ be a Borel-measurable function such that there exists a point $z\in X$ along with sequences of nonnegative real numbers $(a_n),(b_n),(c_n)$ such that
\[
\norm{T_n(z_0,\dots,z_{n-1})-z}\le \max\{a_{n-1}, (1+b_{n-1})\norm{z_{n-1}-z} -c_{n-1}\}
\]
for all $z_0,\dots,z_{n-1}\in X^n$. Fix $X$-valued random variables $x_0$ and $(y_n)$ such that $\EE[y_{n}\mid \mathcal{F}_{n}]=0$ for each $n\in\mathbb{N}$, where $\mathcal{F}_n:= \sigma(x_0,y_0,\dots,y_{n-1})$, i.e.\ $\mathcal{F}_n$ is the $\sigma$-algebra generated by $x_0,y_0,\dots,y_{n-1}$. In $x_0$ and $(y_n)$, define the iteration
\[
x_{n+1}:= T_{n+1}(x_0,\dots,x_{n}) + y_n.
\]
Suppose that $a_n \to 0$, $\sum_{n=0}^\infty b_n<\infty$, $\sum_{n=0}^\infty c^2_n<\infty$ and $\sum_{n=0}^\infty\EE[\norm{y_n}^2]< \infty$ with upper bounds $A,B,C,M>0$ and rates of convergence $\varphi,\beta,\gamma,\mu:(0,\infty)\to\mathbb{N}$, i.e.\
\[
\forall n\geq \varphi(\varepsilon)\left( a_n<\varepsilon\right)\text{ as well as }\sum_{n=\beta(\varepsilon)}^\infty b_n, \sum_{n=\gamma(\varepsilon)}^\infty c^2_n,\sum_{n=\mu(\varepsilon)}^\infty\EE[\norm{y_n}^2]< \varepsilon
\]
for any $\varepsilon>0$. Further, assume $\sum_{n=0}^\infty c_n = \infty$ with a rate of divergence $\theta:\mathbb{N}\times (0,\infty)\to\mathbb{N}$, i.e.\ $\sum_{n=k}^{\theta(k,b)}c_n\geq b$ for any $k\in\mathbb{N}$ and $b>0$. Lastly, assume that $\EE[\norm{x_n}^2]<\infty$ for all $n\in\mathbb{N}$ with upper bounds $L_n>0$. Then $x_n\to z$ a.s.\ with a rate
\[
\rho(\lambda,\varepsilon):= \theta\left(\max\left\{\chi_{\varepsilon/2}\left(\frac{\lambda^2\varepsilon}{4K_{\varepsilon/2}}\right),\varphi(\varepsilon/2)\right\}, \frac{2K_{\varepsilon/2}\sqrt{K_{\varepsilon/2}}(L_{\varphi(\varepsilon/2)}+M_{\varepsilon/2})}{\lambda\varepsilon} \right) 
\]
where, for any $\delta>0$, we define
\begin{gather*}
K_\delta:=(1+B^2)e^{\delta B},\\
M_\delta:=(1+\delta B)C+\delta(1+\delta B)B+M,\\
\chi_\delta(\varepsilon):=\max\left\{\mu\left(\frac{\varepsilon}{3}\right),\gamma\left(\frac{\varepsilon}{3(1+\delta B)}\right),\beta\left(\frac{\varepsilon}{3\delta(1+\delta B)}\right)\right\}.
\end{gather*}
\end{theorem}
\begin{proof}
We reduce the result to the Robbins-Siegmund theorem, following (a slight modification of) the well-known argument given in \cite{robbins-siegmund:71:lemma}. For this, we first fix a $\delta > 0$ and, writing $(x)^+$ for $\max\{x,0\}$, we set
\[
W_{n,\delta} := \left(\left(\norm{x_n-z}- \delta\right)^+\right)^2
\]
as well as the $X$-valued random variables $T_n:= T_n(x_0,\ldots,x_{n-1})-z$ for $n\geq 1$ and
\[
u_n:= T_{n+1}\mathbf{1}_{\norm{T_{n+1}}\leq \delta} +\delta \frac{T_{n+1}}{\norm{T_{n+1}}}\mathbf{1}_{\norm{T_{n+1}}>\delta}
\]
for $n\in\mathbb{N}$. By definition, $T_n$ is $\mathcal{F}_{n-1}$-measurable and so $u_n$ is $\mathcal{F}_n$-measurable. Further, we immediately have $\norm{u_n}\leq\delta$ pointwise everywhere which implies
\[
W_{n+1,\delta}= \left(\left(\norm{x_{n+1}-z}- \delta\right)^+\right)^2\leq \left(\left(\norm{x_{n+1}-z-u_n}+\norm{u_n}- \delta\right)^+\right)^2\leq \norm{x_{n+1}-z-u_n}^2
\]
for all $n\in\mathbb{N}$. Further, note that by definition we have 
\[
T_{n+1}-u_n=\left(1-\frac{\delta}{\norm{T_{n+1}}}\right)T_{n+1}\mathbf{1}_{\norm{T_{n+1}}>\delta}
\]
so that $\norm{T_{n+1}-u_n}=(\norm{T_{n+1}}-\delta)^+$. Combined with the assumption that $\EE[y_n\mid\mathcal{F}_n]=0$, we can then derive that
\begin{align*}
\EE[W_{n+1,\delta}\mid \mathcal{F}_n] &\leq\EE[\norm{x_{n+1}-z-u_n}^2\mid \mathcal{F}_n] \\
&=\EE[\norm{ T_{n+1} + y_n-u_n}^2\mid \mathcal{F}_n]\\
&=\EE[\norm{T_{n+1} - u_n}^2\mid \mathcal{F}_n] + \EE[\norm{y_n}^2\mid\mathcal{F}_n] + 2\EE[\langle y_n, T_{n+1} - u_n\rangle\mid\mathcal{F}_n]\\
&=\norm{T_{n+1} - u_n}^2 + \EE[\norm{y_n}^2\mid\mathcal{F}_n]\\
&=\left(\left(\norm{T_{n+1}}- \delta\right)^+\right)^2  + \EE[\norm{y_n}^2\mid\mathcal{F}_n]
\end{align*}
for all $n\in\mathbb{N}$. Setting $N:= \varphi(\delta)$ yields that $a_{N+n}\leq \delta$ for all $n\in\mathbb{N}$, so that our main assumption on the mappings $T_n$ yield
\begin{align*}
\left(\norm{T_{N+n+1}}- \delta\right)^+ &\le  \max\{a_{N+n} -\delta, (1+b_{N+n})\norm{x_{N+n}-z} -c_{N+n}- \delta,0\}\\
&\le\max\{(1+b_{N+n})\norm{x_{N+n}-z} -c_{N+n}- \delta,0\}\\
&=\left( (1+b_{N+n})(\norm{x_{N+n}-z}- \delta) -c_{N+n} +\delta b_{N+n}\right)^+.
\end{align*}
Now, either we have $\norm{T_{N+n+1}}>\delta$ so that $\left(\norm{T_{N+n+1}}- \delta\right)^+=\norm{T_{N+n+1}}- \delta>0$ and therefore
\begin{align*}
0<\left(\norm{T_{N+n+1}}-\delta \right)^+&\le \left( (1+b_{N+n})(\norm{x_{N+n}-z}- \delta) -c_{N+n} +\delta b_{N+n}\right)^+\\
&=(1+b_{N+n})(\norm{x_{N+n}-z}- \delta) -c_{N+n} +\delta b_{N+n}\\
&\leq (1+b_{N+n})(\norm{x_{N+n}-z}- \delta)^+ -c_{N+n} +\delta b_{N+n}
\end{align*}
which yields
\[
\left(\left(\norm{T_{N+n+1}}-\delta \right)^+\right)^2 \le \left( (1+b_{N+n})(\norm{x_{N+n}-z}- \delta)^+ -c_{N+n} +b_{N+n}\delta\right)^2
\]
in this case. However, if $\norm{T_{N+n+1}}\leq \delta$, then it holds that $\left(\norm{T_{N+n+1}}- \delta\right)^+=0$ and so the above inequality is true unconditionally. Finally, from this we get that
\begin{align*}
\left(\left(\norm{T_{N+n+1}}-\delta\right)^+\right)^2 \le &\left(1+\delta b_{N+n}\right)(1+b_{N+n})^2W_{N+n,\delta} - 2(1+b_{N+n})c_{N+n}\sqrt{W_{N+n,\delta}}\\
&+ \left(1+\delta b_{N+n}\right)c_{N+n}^2 +\delta b_{N+n}\left(1+\delta b_{N+n}\right).
\end{align*}
which we derive utilizing that $(x+y)^2\le (1+y)x^2+y(1+y)$ for any $x\in\mathbb{R}$ and $y\geq 0$, instantiated with
\[
x:= (1+b_{N+n})(\norm{x_{N+n}}- \delta)^+ -c_{N+n}=(1+b_{N+n})\sqrt{W_{N+n,\delta}} -c_{N+n}
\]
as well as $y:= \delta b_{N+n}$, and noting that $1+\delta b_{N+n}\geq 0$.

All in all, we have derived that 
\[
\EE[W_{N+n+1,\delta}\mid \mathcal{F}_{N+n}]\leq (1+\alpha_{N+n,\delta})W_{N+n,\delta} - u_{N+n}\sqrt{W_{N+n,\delta}} + C_{N+n,\delta}
\]
holds for all $n\in\mathbb{N}$, where $u_n:= 2(1+b_n)c_n$ as well as $\alpha_{n,\delta}:=\left(1+\delta b_n\right)(1+b_n)^2-1$ and 
\[
C_{n,\delta}:= \left(1+\delta b_n\right)c_n^2 +\delta b_n\left(1+\delta b_n\right) + \EE[\norm{y_{n}}^2\mid \mathcal{F}_{n}].
\]
It is elementary to verify that, by construction, we have $\sum_{n=0}^\infty \EE[C_{N+n,\delta}]<\infty$ with a bound $M_\delta$ and rate of convergence $\max\{\chi_\delta(\varepsilon)-N,0\}$. Further, we have $\prod_{n=0}^\infty (1+\alpha_{N+n,\delta})<\infty$ with a bound $K_\delta$ and it can be immediately verified that $\theta(N+n,b)-N$ is a rate of divergence for $(u_{N+n})$.

Hence, we can apply Theorem \ref{thm:RS} with $X_n:=W_{N+n,\delta}$, $V_n:=\sqrt{W_{N+n,\delta}}$ and $\mathcal{F}_n:=\mathcal{F}_{N+n}$ (as well as $f(x):=\sqrt{x}$ so that we set $\psi(a):=\sqrt{a}$ and $\kappa(\varepsilon):= \varepsilon^2$ following Example \ref{ex:cisc}, (1)) to derive that $W_{N+n,\delta}\to 0$ a.s.\ with a certain rate $\Delta^\delta(\lambda,\varepsilon)$ arising from Theorem \ref{thm:RS}. At last, let $\varepsilon,\lambda>0$ be given. Then for $\delta=\varepsilon/2$, we obtain thereby that
\begin{align*}
&\PP\left(\exists n\geq\Delta^{\varepsilon/2}(\lambda,\varepsilon^2/4)+N\left( \norm{x_{n}-z} \ge \varepsilon\right)\right)\\
&\qquad=  \PP\left(\exists n\geq\Delta^{\varepsilon/2}(\lambda,\varepsilon^2/4) \left(\left(\left(\norm{x_{N+n}-z}- \varepsilon/2\right)^+\right)^2 \ge \varepsilon^2/4\right)\right)\\
&\qquad=  \PP\left(\exists n\geq\Delta^{\varepsilon/2}(\lambda,\varepsilon^2/4) \left(W_{N+n,\varepsilon/2} \ge \varepsilon^2/4\right)\right)<\lambda
\end{align*}
and so $\norm{x_n}\to 0$ a.s.\ with a rate given by 
\[
\rho(\lambda,\varepsilon)= \Delta^{\varepsilon/2}(\lambda,\varepsilon^2/4)+N=\Delta^{\varepsilon/2}(\lambda,\varepsilon^2/4)+\varphi(\varepsilon/2).
\]
The rate presented in the above theorem follows from simplifying the corresponding expressions arising from Theorem \ref{thm:RS}.
\end{proof}

It should be observed that there exist alternative proofs of Dvoretzky's theorem, notably that of Derman and Sacks \cite{DS1959}, which are more closely tailored to the specific assumptions of the theorem and do not make direct use of supermartingale convergence. A careful analysis of such proofs might result in a different rate of convergence and thus an alternative quantitative version of Dvoretzky's theorem to the one given here.\footnote{This observation is due to R. Arthan and P. Oliva (private communication), who are working on an analysis of the proof due to Derman and  Sacks.}

\section{Stochastic quasi-Fej\'er monotonicity in the presence of uniqueness}\label{sec:fejer}

We now consider one last abstract scenario commonly used to analyse stochastic approximation methods, for which we fix the following setup: For the rest of this section, let $(X,d)$ be a metric space, and consider the problem of finding a zero $F(z)=0$ of some general measurable function $F:X\to [0,\infty]$, assuming that the set of zeros $\mathrm{zer}F$ is nonempty.

We consider the general class of methods that are \emph{stochastically quasi-Fej\'er monotone}, that is sequences of $X$-valued random variables $(x_n)$ which satisfy the descent condition
\[
\EE[d(x_{n+1},z)\mid\mathcal{F}_n]\leq (1+\zeta_n)d(x_n,z)+\xi_n\text{ a.s.}
\]
for all $n\in\mathbb{N}$ and $z\in \mathrm{zer}F$, where $(\mathcal{F}_n)$ is a given filtration and $(\zeta_n),(\xi_n)$ are sequences of real-valued random variables with $\sum_{n=0}^\infty \xi_n,\sum_{n=0}^\infty \zeta_n<\infty$ a.s.

The notion of (quasi-)Fej\'er monotonicity plays a central role in the study of iterative methods in nonlinear analysis and optimization, and we refer to the well-known expositions in \cite{BC2017,Combettes2001,Com2009}, along with the more recent work in the deterministic setting \cite{KLN2018,KLAN2019,KP2025,Pischke2025a} (see also \cite{Pis2023}). The stochastic variants of quasi-Fej\'er monotonicity seem to have been first considered in the pioneering works of Ermol'ev \cite{Erm1969,Erm1971,ET1973}, which were much later refined and generalised from Euclidean spaces to separable Hilbert spaces \cite{CP2015,CP2019} (and recently, to nonlinear Hadamard spaces \cite{Pischke2026}). 

In order to induce convergence for Fej\'er monotone sequences, one normally requires some structure on the solution set $\mathrm{zer}F$, and the associated level sets of $F$ representing approximation solutions, such as (weak) compactness. In a situation that is broadly analogous to that of the Robbins-Siegmund theorem, effective ``full'' rates of convergence are not possible at this level of generality.\footnote{Recall footnote \ref{specker}, and further see the discussions in \cite{KLN2018,KLAN2019}. Quantitative results in such general circumstances, as given in \cite{KLN2018} for deterministic Fej\'er monotone sequences and by the authors in \cite{NeriPischkePowell2026} for a stochastic setting, again only provide weaker generalized oscillation bounds.} However, in many applications the solution set is actually a \emph{singleton}, and the main result of this section is that in this situation, rates of convergence can be explicitly constructed whenever the method produces arbitrarily good approximate solutions (as is widely the case) and a stochastic variant of the uniqueness of the solution can be quantitatively witnessed through a function $\tau:(0,\infty)\to (0,\infty)$ satisfying
\[
\forall x\in D\ \forall \varepsilon>0\left( \EE[F(x)]<\tau(\varepsilon)\to \EE[d(x,z)]<\varepsilon\right),
\]
where $D$ is a given collection of $X$-valued random variables. In this case we say that $F$ has a strongly unique zero in expectation. The corresponding modulus $\tau$ can hence be seen as an instantiation of our previous abstract notion of a regularity modulus, connecting the main sequences whose convergence we wish to establish, here $d(x_n,z)$, with an auxiliary one, here $F(x_n)$, so that we only require a weak approximation property that latter sequence to guarantee convergence.

In a deterministic setting, similar general quantitative considerations for Fej\'er monotone sequences are made in \cite{KLAN2019} (see also \cite{Pis2023,Pischke2025a}). In the light of that work, this notion strong uniqueness in expectation is a stochastic analog of the closely related notion of a modulus of uniqueness as considered in \cite{KLAN2019}, which is a special case of the abstract notion of a modulus of regularity introduced therein, that is a modulus $\tau:(0,\infty)\to (0,\infty)$ with the property that
\[
\forall x\in \overline{B}_b(z)\ \forall \varepsilon>0\left( F(x)<\tau(\varepsilon)\to \mathrm{dist}_{\mathrm{zer}F}(x)<\varepsilon\right),
\]
given some suitable region $\overline{B}_b(z)$ around a fixed solution $z$. This latter form of regularity covers essentially all known regularity-type notions from the literature of deterministic optimization (see \cite{KLAN2019}), and is moreover optimal in that it can be shown that already for the deterministic Picard iteration, a corresponding (deterministic) modulus of regularity can be constructed from an associated (uniform) rate of convergence for the process towards a solution (see Proposition 4.4 in \cite{KLAN2019}). Lifting this general regularity notion and its associated results on rates of convergence to a stochastic setting is far from trivial and forms the subject of a separate paper \cite{PischkePowell2026} by the second and third author, which in particular requires a careful use of measurable selection theory and a (slightly) more restricted notion of stochastic quasi-Fej\'er monotonicity. 

However, convergence of stochastic quasi-Fej\'er monotone sequences in the case that $F$ enjoys strong uniqueness in expectation, a setting which still covers a considerable range of problems (cf.\ Section \ref{sec:RM} below), can be dealt with immediately within the framework of the current paper, as we now show.

It turns out that most properties of the metric are inessential for our main arguments, and so we follow the approach of the recent work \cite{Pischke2025a} and consider more general distance functions $\phi:X\times X\to [0,\infty)$, which allow us to capture an even wider class of iterations. The main assumption we need to place on $\phi$ is the following property (also introduced in \cite{Pischke2025a}) that connects it back to the metric.

\begin{definition}
A mapping $\phi:X\times X\to [0,\infty)$ is called uniformly consistent if there exists a function $\kappa:(0,\infty)\to (0,\infty)$ such that
\[
\forall \varepsilon>0\ \forall x,y\in X\left(\phi(x,y)<\kappa(\varepsilon)\rightarrow d(x,y)<\varepsilon\right).
\]
Such a $\kappa$ is then called a modulus of uniform consistency for $\phi$.
\end{definition}

Such generalized distance functions in particular unify perturbations of the metric, like metric powers, with notions such as Bregman distances, and we refer to \cite{Pischke2025a} as well as \cite{PischkePowell2026} (and to recent applications \cite{Pis2025b,PK2024} featuring this or similar notions) for further illustrations on the breadth of this perspective. In the context of this modification of the distance function, we consider corresponding relativizations of the strong uniqueness in expectation as well as the stochastic quasi-Fej\'er monotonicity.

\begin{definition}
Let $D$ be a collection of $X$-valued random variables. Let $\phi:X\times X\to [0,\infty)$ be a mapping which is measurable in its first argument. We say that the zero $z$ of $F$ is strongly $\phi$-unique in expectation (over $D$) if there exists a function $\tau:(0,\infty)\to (0,\infty)$ such that
\[
\forall x\in D\ \forall \varepsilon>0\left( \EE[F(x)]<\tau(\varepsilon)\to \EE[\phi(x,z)]<\varepsilon\right).
\]
Such a $\tau$ is then called a modulus of strong $\phi$-uniqueness in expectation for $F$ (over $D$).
\end{definition}

\begin{definition}
Let $\phi:X\times X\to [0,\infty)$ be a mapping which is measurable in its first argument. Let $(\mathcal{F}_n)$ be a filtration and let $(x_n)$ be a sequence of $X$-valued $\mathcal{F}_n$-measurable random variables. Then $(x_n)$ is called stochastically $\phi$-quasi-Fej\'er monotone w.r.t.\ $S\subseteq X$ and $(\mathcal{F}_n)$ if
\[
\EE[\phi(x_{n+1},z)\mid\mathcal{F}_n]\leq (1+\zeta_n)\phi(x_n,z)+\xi_n\text{ a.s.}
\]
for all $z\in S$ and all $n\in\mathbb{N}$, where $(\xi_n),(\zeta_n)$ are suitable sequences of nonnegative, integrable real-valued random variables.
\end{definition}

In analogy to the non-stochastic case (see in particular \cite{KLAN2019}), the only further property needed to induce convergence is that the iteration has approximate zeros for the function $F$ infinitely often, which now needs to hold in expectation: 

\begin{definition}
Let $(x_n)$ be a sequence of $X$-valued random variables. We say that $(x_n)$ has the $\liminf$-property in expectation relative to $F$ if $\liminf_{n\to\infty}\EE[F(x_n)]=0$. A function $\varphi:(0,\infty)\times\mathbb{N}\to (0,\infty)$ witnessing this property quantitatively in the sense that
\[
\forall \varepsilon>0\ \forall N\in\mathbb{N}\ \exists n\in [N;\varphi(\varepsilon,N)]\left( \EE[F(x_n)]<\varepsilon\right)
\]
is called a $\liminf$-bound in expectation for $(x_n)$ relative to $F$.
\end{definition}

Under these main assumptions, together with some minor assumptions on surrounding (quantitative) data, we obtain the following result on the convergence of stochastically quasi-Fej\'er monotone sequences in the presence of uniqueness as a corollary of Theorem \ref{thm:RatesGeneral}.

\begin{theorem}\label{thm:Fejer}
Let $(X,d)$ be a metric space and let $\phi:X\times X\to [0,\infty)$ be a mapping which is measurable in its first argument. Let $F:X\to [0,\infty]$ be measurable with $\mathrm{zer}F=\{z\}$, where the zero $z$ is strongly $\phi$-unique in expectation over a collection $D$ of $X$-valued random variables with a modulus $\tau:(0,\infty)\to (0,\infty)$, i.e.
\[
\forall x\in D\ \forall \varepsilon>0\left( \EE[F(x)]<\tau(\varepsilon)\to \EE[\phi(x,z)]<\varepsilon\right).
\]
Let $(\mathcal{F}_n)$ be a filtration and let $(x_n)\subseteq D$ be a sequence of $X$-valued $\mathcal{F}_n$-measurable random variables such that $\phi(x_n,z)$ is integrable for all $n\in\mathbb{N}$. Suppose further that $(x_n)$ is stochastically $\phi$-quasi-Fej\'er monotone w.r.t.\ $\mathrm{zer}F$ and $(\mathcal{F}_n)$, i.e.\
\[
\EE[\phi(x_{n+1},z)\mid\mathcal{F}_n]\leq (1+\zeta_n)\phi(x_n,z)+\xi_n\text{ a.s.}
\]
for all $n\in\mathbb{N}$, where $(\xi_n),(\zeta_n)$ are sequences of integrable real-valued random variables such that there exist $K>0$ and $\chi:(0,\infty)\to\mathbb{N}$ with $\prod_{n=0}^\infty (1+\zeta_n)<K$ a.s.\ and $\sum_{n=\chi(\varepsilon)}^\infty \EE[\xi_n]<\varepsilon$ for all $\varepsilon>0$. Suppose furthermore that $(x_n)$ has the $\liminf$-property in expectation relative to $F$ with a bound $\varphi:(0,\infty)\times\mathbb{N}\to (0,\infty)$, i.e.
\[
\forall \varepsilon>0\ \forall N\in\mathbb{N}\ \exists n\in [N;\varphi(\varepsilon,N)]\left( \EE[F(x_n)]<\varepsilon\right).
\]
Then $\EE[\phi(x_n,z)]\to 0$ with rate
\[
\rho(\varepsilon):=\varphi\left(\tau\left(\frac{\varepsilon}{2K}\right),\chi\left(\frac{\varepsilon}{2K}\right)\right)
\]
and $\phi(x_n,z)\to 0$ a.s.\ with rate $\rho'(\lambda,\varepsilon):=\rho(\lambda\varepsilon)$. If furthermore $\phi$ is uniformly consistent with modulus $\kappa:(0,\infty)\to (0,\infty)$, i.e.
\[
\forall \varepsilon>0\ \forall x,y\in X\left(\phi(x,y)<\kappa(\varepsilon)\rightarrow d(x,y)<\varepsilon\right),
\]
then $d(x_n,z)\to 0$ a.s.\ with rate $\rho'(\lambda,\kappa(\varepsilon))$.
\end{theorem}
\begin{proof}
Define $X_{n}:=\phi(x_n,z)$ for any $n\in\mathbb{N}$. As $(x_n)$ is adapted to $(\mathcal{F}_n)$ and $\phi$ is measurable in its first argument, also $(X_{n})$ is adapted to $(\mathcal{F}_n)$. As in the discussion after Theorem \ref{thm:RatesGeneral}, by the assumptions on $\varphi$ and $\tau$, we have that $\varphi(\tau(\varepsilon),N)$ is a $\liminf$-modulus for $\EE[X_{n}]=\EE[\phi(x_n,z)]$. For $A_n:=\zeta_n$ and $C_n:=\xi_n$ as well as $f:=\mathrm{id}$, Theorem \ref{thm:RatesGeneral} then yields the respective rates for $\EE[\phi(x_n,z)]\to 0$ and $\phi(x_n,z)\to 0$ a.s. The latter in particular means 
\[
\forall \varepsilon,\lambda>0\left(\PP(\exists n\geq \rho'(\lambda,\varepsilon)(\phi(x_n,z)\geq\varepsilon))<\lambda\right).
\]
If $\phi$ is then also uniformly consistent, we further have for any $\varepsilon,\lambda>0$ that
\[
\PP(\exists n\geq \rho'(\lambda,\kappa(\varepsilon))(d(x_n,z)\geq \varepsilon))\leq \PP(\exists n\geq \rho'(\lambda,\kappa(\varepsilon))(\phi(x_n,z)\geq\kappa(\varepsilon)))<\lambda
\]
so that $\rho'(\lambda,\kappa(\varepsilon))$ is a rate for $d(x_n,z)\to 0$.
\end{proof}

In analogy to Corollary \ref{cor:qualCorMain}, we also highlight the corresponding qualitative theorem obtained from the above result by disregarding the quantitative information:

\begin{corollary}\label{cor:qualCorFejer}
Given a metric space $(X,d)$ and functions $\phi:X\times X\to [0,\infty)$ and $F:X\to [0,\infty]$, the former being measurable in its first argument and the latter being measurable, suppose that $F$ has a unique zero $z$ which is strongly $\phi$-unique in expectation over a collection $D$ of $X$-valued random variables. Suppose that $(x_n)\subseteq D$ is a sequence of $X$-valued random variables which is stochastically $\phi$-quasi-Fej\'er monotone w.r.t.\ $\mathrm{zer}F$ (and a suitable filtration) and that it has the $\liminf$-property in expectation relative to $F$. Then $\EE[\phi(x_n,z)]\to 0$ and $\phi(x_n,z)\to 0$ a.s. If $\phi$ is uniformly consistent, then $d(x_n,z)\to 0$ a.s.
\end{corollary}

Just as with our quantitative Robbins-Siegmund theorem, refined results providing faster rates of convergence for stochastically Fej\'er monotone sequences under stronger assumptions are easy to produce. To give one example, the following result in the context of linear moduli of strong uniqueness is an immediate corollary of Theorem \ref{thm:RSfast} (and we omit the proof): 

\begin{theorem}\label{thm:fejerFast}
Let $(X,d)$ be a metric space and let $\phi:X\times X\to [0,\infty)$ be a mapping which is measurable in its first argument. Let $F:X\to [0,\infty]$ be measurable with $\mathrm{zer}F=\{z\}$, where the zero $z$ is strongly $\phi$-unique in expectation over a collection $D$ of $X$-valued random variables where further $\EE[F(x)]\geq t\EE[\phi(x,z)]$ for all $x\in D$ with $t>0$. Let $(\mathcal{F}_n)$ be a filtration and let $(x_n)\subseteq D$ be a sequence of $X$-valued $\mathcal{F}_n$-measurable random variables such that $\phi(x_n,z)$ is integrable for all $n\in\mathbb{N}$. Suppose further that $(x_n)$ satisfies
\[
\EE[\phi(x_{n+1},z)\mid\mathcal{F}_n]\leq (1+\zeta_n)\phi(x_n,z)-\eta_n F(x_n)+\xi_n\text{ a.s.}
\]
for all $n\in\mathbb{N}$, where $(\zeta_n)$, $(\eta_n)$ are sequences of nonnegative reals and $(\xi_n)$ is a sequence of integrable real-valued random variables such that 
\[
\EE[\xi_n]\leq d/(n+r)^2\ \text{ and }\ \zeta_n+c/(n+r)\leq t\eta_n
\]
for all $n\in\mathbb{N}$ where $c>1$, $d\geq 0$ and $r\in\mathbb{N}\setminus\{0\}$. Finally suppose that $K\geq 1$ and $L>0$ are such that $\prod_{i=0}^\infty (1+\zeta_i)<K$ and $L\geq \EE[\phi(x_0,z)]$. Then
\[
\EE[\phi(x_n,z)]\leq\frac{u}{n+r} \text{ and }\PP\left(\exists m\geq n(\phi(x_m,z)\geq \varepsilon)\right)\leq \frac{1}{\varepsilon}\cdot\frac{K(u+2d)}{n+r}
\]
for all $n\in\mathbb{N}$ and $\varepsilon>0$, where $u\geq \max\left\{d/(c-1),rL\right\}$.
\end{theorem}

We end this section by observing the following result derived from the previous Proposition \ref{pro:AStoEconvex}, which allows us to recognize suitable moduli witnessing a type of pointwise strong uniqueness as moduli of strong uniqueness in mean, which is particularly useful for deriving concrete examples of such moduli.

\begin{proposition}\label{pro:FejModUniqConv}
Let $D$ be a collection of $X$-valued random variables. Let $\phi:X\times X\to [0,\infty)$ be a mapping which is measurable in its first argument and let $F:X\to [0,\infty]$ be measurable with $z\in\mathrm{zer}F$ and such that $\phi(x,z)$ is integrable and
\[
F(x)\geq \tau(\phi(x,z))\text{ a.s.}
\]
for all $x\in D$, where $\tau:[0,\infty)\to [0,\infty)$ is a convex and strictly increasing function. Then $z$ is strongly $\phi$-unique in expectation over $D$ and $\tau$ is a modulus. 
\end{proposition}

In both cases, note that the above requirement that $F(x)\geq \tau(\phi(x,z))$ a.s.\ for all $x\in D$ is in particular true if $F(x)\geq \tau(\phi(x,z))$ for all $x\in X$.

\section{Stochastic approximation}\label{sec:RM}

We conclude with a simple example which illustrates how our abstract results may apply in practice. We deliberately focus on a widely known method, the Robbins-Monro procedure, which while of central importance in modern statistics, is intended to be merely illustrative. Indeed, as highlighted in the introduction, many further applications of our results have already been derived (we refer again to \cite{Pischke2025b,Pischke2026} and \cite{PischkePowell2026}), and many further potential applications are easy to identify. In particular, in the context of the previous section, concrete scenarios that can be formulated as root finding problems $F(z)=0$ whose solution is strongly unique in expectation include fixed-point problems for quasi-contractions, minimization problems for uniformly quasi-convex functions, and set-valued inclusion problems for uniformly accretive or monotone operators. In all these cases, a corresponding regularity modulus $\tau$ concretely arises from a simple pointwise condition of the form $F(x)\geq \tau(\phi(x,z))$, which can be converted to strong uniqueness in mean via Proposition \ref{pro:FejModUniqConv} above. The reader is directed to \cite{PischkePowell2026} for more explicit and extensive discussions in this vein, on both notions of regularity and the many concrete forms they take, along with the variety of different methods that are encompassed by the corresponding notion of stochastic quasi-Fej\'er monotonicity, as in particular illustrated by the applications presented therein. 

Background on the Robbins-Monro scheme and the many well-known algorithms connected to it can be found in any standard reference for stochastic approximation (e.g.\ \cite{Duflo97,Kushner2003,Lai2003}). Here we consider a general scheme in Hilbert spaces, motivated by the early works \cite{Revesz73,Salov80,Yin90} on the Robbins-Monro method in such contexts, that matches our level of abstraction, and in line with our overall approach, derive for it a quantitative convergence theorem with the corresponding rate formulated in terms of abstract moduli. While we later rederive standard convergence rates, our general convergence result is new at this level of abstraction, and covers a broad family of stochastic approximation algorithms.

Now, letting $H$ be an arbitrary separable Hilbert space, let $(x_n)$ be the sequence of $H$-valued random variables recursively defined by
\[
x_{n+1}:=x_n-a_ny_n \label{RM}\tag{RM}
\]
from some initial $x_0\in H$, where $(y_n)$ another sequence of $H$-valued random variables and $(a_n)$ a sequence of nonnegative reals. The classic Robbins-Monro scheme for finding the root of a function $M:\RR\to \RR$ is then an instance of $\eqref{RM}$ for $H=\RR$ and $y_n=M(x_n)+\varepsilon_n$ for some noise terms $(\varepsilon_n)$.

Any one of our abstract convergence frameworks, that is Robbins-Siegmund, Dvoretzky or Fej\'er monotonicity, could now be used to set up an abstract quantitative convergence result for \eqref{RM}. We choose the more abstract setting of the first. In the resulting Theorem \ref{res:RM} given below, conditions (i) -- (iii) in particular abstract standard assumptions (similar as e.g.\ done in \cite{Salov80}) that would normally be imposed on the objective function $M$ and the stochastic noise $(\varepsilon_n)$, with (iii) in particular implicitly representing a type of regularity assumption, while (iv) represents a quantitative rendering of the classic conditions $\sum_{n=0}^\infty a_n=\infty$ and $\sum_{n=0}^\infty a_n^2<\infty$. For details on how standard concepts like conditional expectations extend to Hilbert spaces, the reader is directed to \cite{ledoux:91:banach}.

\begin{theorem}\label{res:RM}
Suppose that $(x_n)$, $(y_n)$ and $(a_n)$ satisfy \eqref{RM}. Let $z\in H$ and suppose that $L>0$ satisfies $\norm{x_0-z}^2<L$. Define $\mathcal{F}_n:=\sigma(x_0,y_0,\dots,x_{n-1},y_{n-1})$, i.e.\ $\mathcal{F}_n$ is the $\sigma$-algebra generated by $x_0,y_0,\dots,x_{n-1},y_{n-1}$, and suppose that the following conditions are satisfied:
\begin{enumerate}
\item[(i)] there exist $c,d>0$ and a sequence of nonnegative random variables $(d_n)$ satisfying $\sup_{n\in\NN}\EE[d_n]\leq d$ such that 
\[
\EE[\norm{y_n}^2\mid\mathcal{F}_n]\leq c\norm{x_n-z}^2+d_n\text{ a.s.}
\]
for all $n\in\NN$;
\item[(ii)] $\langle x_n-z,\EE[y_n\mid\mathcal{F}_n]\rangle \geq 0$ for all $n\in\NN$;
\item[(iii)] there exists a function $\tau:(0,\infty)\to (0,\infty)$ such that 
\[
\EE[\langle x_n-z,\EE[y_n\mid\mathcal{F}_n]\rangle]<\tau(\varepsilon)\to \EE[f(\norm{x_n-z}^2)]<\varepsilon
\]
for all $\varepsilon>0$ and $n\in\NN$, where $f$ is a s.i.c.c.\ function with moduli $\psi$ and $\kappa$;
\item[(iv)] $\sum_{n=0}^\infty a_n=\infty$ with rate of divergence $\theta$ and $\sum_{n=0}^\infty a_n^2<M$ with rate of convergence $\chi$, i.e.\ $\sum_{n=k}^{\theta(k,b)}a_n\geq b$ and $\sum_{n=\chi(\varepsilon)}^\infty a_n^2< \varepsilon$ for any $\varepsilon,b>0$ and any $k\in\mathbb{N}$.
\end{enumerate}
Then $\EE[f(\norm{x_n-z}^2)]\to 0$ with rate 
\[
\rho(\varepsilon):=\theta\left(\chi\left(\frac{\kappa\left(K_1\varepsilon\right)}{d}\right),\frac{K_2}{\tau(K_1\varepsilon)}\right)
\]
for $K_1:=\frac{1}{2}\psi\left(e^{-cM}\right)$ and $K_2:=e^{cM}(L+dM)$ as well as $x_n\to z$ a.s.\ with rate $\rho'(\lambda,\varepsilon):=\rho(\lambda f(\varepsilon^2))$.
\end{theorem}

\begin{proof}
We first observe that $x_n$ is $\mathcal{F}_n$ measurable for all $n\in\NN$, and thus
\begin{align*}
\EE[\norm{x_{n+1}-z}^2\mid\mathcal{F}_{n}]&=\EE[\norm{x_n-a_ny_n-z}^2\mid \mathcal{F}_n]\\
&=\EE[\norm{x_n-z}^2\mid\mathcal{F}_n]-2a_n\EE[\langle x_n-z,y_n\rangle\mid \mathcal{F}_n]+a_n^2\EE[\norm{y_n}^2\mid\mathcal{F}_n]\\
&=\norm{x_n-z}^2-2a_n\langle x_n-z,\EE[y_n\mid \mathcal{F}_n]\rangle+a_n^2\EE[\norm{y_n}^2\mid\mathcal{F}_n]
\end{align*}
almost surely for any $n\in\mathbb{N}$. Using (i), we therefore have
\[
\EE[\norm{x_{n+1}-z}^2\mid\mathcal{F}_{n}]\leq(1+ca_n^2)\norm{x_n-z}^2-2a_n\langle x_n-z,\EE[y_n\mid\mathcal{F}_n]\rangle+a_n^2d_n\text{ a.s.}
\]
for any $n\in\mathbb{N}$. We now apply Theorem \ref{thm:RS} with $X_n:=\norm{x_{n}-z}^2$ as well as $V_n:=\langle x_n-z,\EE[y_n\mid\mathcal{F}_n]\rangle$, by which the given rates follow. For that, we simply note that nonnegativity of $V_n$ follows from (ii) and we have $\prod_{i=0}^\infty (1+ca_n^2)< e^{cM}$ as well as $\sum_{n=k}^{\theta(k,b/2)}2a_n\geq 2\sum_{n=k}^{\theta(k,b/2)}a_n \geq b$ and $\sum_{n=\chi(\varepsilon/d)}^\infty \EE[a_n^2d_n]\leq d\sum_{n=\chi(\varepsilon/d)}^\infty a_n^2<\varepsilon$.
\end{proof}

Setting $f=\sqrt{\cdot}$ gives the following corollary illustrating convergence in mean:

\begin{corollary}\label{res:RM:f:is:id}
Assuming that condition (iii) of Theorem \ref{res:RM} is simplified to
\[
\EE[\langle x_n-z,\EE[y_n\mid\mathcal{F}_n]\rangle]<\tau(\varepsilon)\to \EE[\norm{x_n-z}]<\varepsilon,
\]
the conclusions of Theorem \ref{res:RM} simplify to $\EE[\norm{x_n-z}]\to 0$ with rate 
\[
\rho(\varepsilon):=\theta\left(\chi\left(\frac{K^2_1\varepsilon^2}{d}\right),\frac{K_2}{\tau(K_1\varepsilon)}\right)
\]
for $K_1:=\frac{1}{2}\sqrt{e^{-cM}}$, $K_2:=e^{cM}(L+dM)$ as well as $x_n\to z$ a.s.\ with rate $\rho'(\lambda,\varepsilon):=\rho(\lambda \varepsilon)$.
\end{corollary}

For an example of how conditions (i)--(iii) are naturally satisfied at this level of abstraction, let $M:H\to 2^H$ be a set-valued operator, and $z$ a solution to the inclusion problem $0\in M(z)$. Let $y_n$ be an estimator for an element of $M(x_n)$, which we represent abstractly via the condition
\[
\EE[y_n\mid\mathcal{F}_n]\in M(x_n).
\]
If $M$ is monotone, then condition (ii) is immediately satisfied. If $M$ is further $\tau$-uniformly monotone at $z$, i.e.\ $M$ satisfies\footnote{This assumption is a special case of the notion of uniform monotonicity (see e.g.\ \cite{BC2017}) restricted to a particular zero at hand. A similar such assumption is already studied in \cite{KohlenbachKoutsoukouArgyraki2015}.}
\[
\langle x-z,u\rangle\geq \tau(\norm{x-z}) \text{ for all } (x,u)\in M,
\]
with a corresponding modulus $\tau:[0,\infty)\to [0,\infty)$, which is convex, strictly increasing and satisfies $\tau(0)=0$, then Proposition \ref{pro:AStoEconvex} further guarantees our implicit regularity property (iii) (in the simplified form of Corollary \ref{res:RM:f:is:id}). Further, condition (i) is naturally satisfied under a standard linear boundedness condition on $M$ (see already \cite{Revesz73} as well as \cite{Lai2003,Yin90}), that is if we assume that 
\[
\norm{y}\leq c(\norm{x-z}+1)\text{ for all }x\in H\text{ and all }y\in M(x).
\]
Therefore, in the context of a parameter sequence $(a_n)$ satisfying condition (iv), both $\EE[\norm{x_n-z}]\to 0$ and $x_n\to z$ a.s.\ follow from Corollary \ref{res:RM:f:is:id} with corresponding rates as indicated. 

For the special, but still widely studied, particular case that $M$ is $\beta$-strongly monotone in the sense that 
\[
\langle x-y,u-v\rangle\geq \beta \norm{x-y}^2\text{ for all }(x,u),(y,v)\in M,
\]
a simple adaptation of the proof of Theorem \ref{res:RM} puts us in the scope of fast rates as in Theorem \ref{thm:RSfast}: As in the proof of Theorem \ref{res:RM}, we have
\[
\EE[X_{n+1}\mid\mathcal{F}_{n}]\leq (1+ca_n^2)X_n-2a_nV_n+a_n^2d_n\text{ a.s.}
\]
for all $n\in\mathbb{N}$, where $X_n:=\norm{x_n-z}^2$ and $V_n:=\langle x_n-z,\EE[y_n\mid\mathcal{F}_n]\rangle$, but now, using the $\beta$-strong monotonicity of $M$, we can derive the stronger (linear) relation that $V_n\geq \beta X_n$ holds almost surely, and so in particular that $\EE[V_n]\geq \beta\EE[X_n]$ for all $n\in\mathbb{N}$. Defining
\[
a_n:=\frac{1}{\beta(n+r)} \ \ \ \mbox{for} \ \ \  r\geq \frac{2c}{\beta^2},
\]
it is easy to show that $\EE\left[a_n^2d_n\right]\leq d/\beta^2(n+r)^2$ and $ca_n^2+3/2(n+r)\leq 2a_n\beta$ and so Theorem \ref{thm:RSfast} applies, yielding in particular
\[
\EE\left[\norm{x_n-z}^2\right]\leq \frac{u}{n+r}
\]
for all $n\in\mathbb{N}$ where $u\geq \max\left\{2d/\beta,rL\right\}$. In this way, we rederive the well-known asymptotic estimate $\EE[\norm{x_n-z}]=\mathcal{O}(1/\sqrt{n})$ for Robbins-Monro schemes in such contexts.

These results then immediately apply to stochastic subgradient methods for suitable convex functions on $H$. Concretely, whenever $f:H\to \RR$ is a $\tau$-uniformly convex function, i.e.
\[
f(\lambda x+(1-\lambda )y)\leq \lambda f(x)+(1-\lambda)f(y)-\lambda(1-\lambda)\tau(\norm{x-y})
\]
for all $x,y\in H$ and $\lambda\in [0,1]$. In that case, it is well-known (see e.g.\ \cite{BC2017}) that the subderivative $\partial f:H\to 2^H$ of $f$ is $2\tau$-uniformly monotone, and so the corresponding stochastic subgradient method
\[
x_{n+1}:=x_n-a_ny_n \ \ \ \mbox{for} \ \ \ \EE[y_n\mid\mathcal{F}_n]\in\partial f(x_n)\label{SG}\tag{SG}
\]
reduces as usual to \eqref{RM} for $M=\partial f$, converging in mean and almost surely to the minimizer $z$ of $f$, with corresponding effective rates, under the assumption of linear boundedness for the subgradient. Analogous fast rates apply in the special case that $f$ is $\beta$-strongly convex can be derived from the above as well.

\begin{remark}\label{rem:blum}
Alternative regularity conditions commonly encountered in stochastic approximation, such as variants of Blum's condition \cite{blum:54:stochastic}, can similarly be approached using the present abstract framework. In the context of the simple scheme 
\[
x_{n+1}:=x_n-a_n(M(x_n)+\varepsilon_n)
\]
for an operator $M:\mathbb{R}\to\mathbb{R}$, Blum's takes the following form (cf.\ \cite{Lai2003}): For all $0<\varepsilon<1$ there exists $\delta>0$ such that
\[
\inf_{\varepsilon\leq |x-z|\leq\varepsilon^{-1}}\left\{M(x)(x-z)\right\}\geq \delta. \mbox{ a.s.}.
\]
A modulus of uniqueness $\tau$ in our sense can be constructed from a function $\delta(\varepsilon)$ that witnesses the above property under suitable additional conditions, in particular using a quantitative rendering of an associated uniform integrability condition as introduced in \cite{PischkePowell2024}. We refer to the arXiv version \cite{NeriPischkePowell2025b} (cf.\ Section 6.3) of a previous, more extensive variant of the present paper for a more detailed discussion of such results.
\end{remark}

\noindent
{\bf Acknowledgments:} The authors want to thank Miroslav Ba\v{c}\'ak and Paulo Oliva for helpful discussions on the topic of this paper. The first author was partially supported by the EPSRC Centre for Doctoral Training in Digital Entertainment EP/L016540/1, and the third author was partially supported by the EPSRC grant EP/W035847/1.

\bibliographystyle{plain}
\bibliography{ref}

\end{document}